\documentclass[11pt, reqno]{amsart}
\usepackage{amsfonts}
\usepackage{latexsym}
\usepackage{amssymb}
\usepackage{amsmath}
\usepackage{color}
\usepackage{bbm}
\usepackage{tikz}
\usepackage{enumerate}
\usepackage{enumitem}
\usepackage{mathrsfs}
\usepackage{todonotes}
\usepackage{relsize}
\usepackage{xfrac}
\usepackage{verbatim}
\usepackage[left=2.5cm, top=2.5cm,bottom=2.5cm,right=2.5cm]{geometry}

\newcommand{\R}{\mathbb R}
\newcommand{\N}{\mathbb N}

\newcommand{\E}{\mathbb E}
\newcommand{\Pro}{\mathbb P}

\newcommand{\vol}{\mathrm{vol}}

\def\dint{\textup{d}}
\newcommand{\SSS}{\ensuremath{{\mathbb S}}}

\newcommand{\B}{\ensuremath{{\mathbb B}}}

\definecolor{kgreen}{rgb}{0, 0.42, 0}

\newtheorem{thm}{Theorem}[section]

\newtheorem{lemma}[thm]{Lemma}

\newtheorem{proposition}[thm]{Proposition}

{
\theoremstyle{definition}

\newtheorem{rmk}[thm]{Remark}

}
\def\bB{\mathbf{B}}
\def\bC{\mathbf{C}}

\def\bN{\mathbf{N}}

\def\bR{\mathbf{R}}

 	\definecolor{Kgreen}{rgb}{0.0, 0.5, 0.0}

\usepackage{nomencl}
\makenomenclature

\begin{document}


\title[Sharpening the probabilistic Arithmetic-Geometric Mean Inequality]{Sharpening the probabilistic\\ Arithmetic-Geometric Mean Inequality}

\author[Tom Kaufmann]{Tom Kaufmann}
\author[Christoph Thäle]{Christoph Thäle}
\address{Tom Kaufmann: Faculty of Mathematics, Ruhr University Bochum, Germany} \email{tom.kaufmann@rub.de}
\address{Christoph Thäle: Faculty of Mathematics, Ruhr University Bochum, Germany} \email{christoph.thaele@rub.de}

%
%
%
%
\begin{abstract}
We consider the $p$-generalized arithmetic-geometric mean inequality for vectors chosen randomly from the $\ell_p^n$-ball in $\R^n$. In this setting the inequality can be improved or reversed up to a respective scalar constant with high probability, and central limit theorems and large deviation results with respect to this constant have been shown. We sharpen these large deviation results in the spirit of Bahadur and Ranga Rao, thereby providing concrete and asymptotically exact estimates on a non-logarithmic scale for the probability of the inequality being improvable or reversible up to a constant, respectively.  
\end{abstract}

\maketitle

\section{Introduction and Main Results}\label{sec:Introduction}
%
For $n\in \N$ and $x_1, \ldots, x_n \in \R^n$ the arithmetic-geometric mean (AGM) inequality states that
\begin{equation*}
{\left(\prod_{i=1}^n |x_i| \right)}^{1/n} \le \quad  { \frac{1}{n} \sum_{i=1}^n |x_i|}.
\end{equation*}
Additionally, for $p>0$ the $p$-generalized arithmetic-geometric mean ($p$-AGM) inequality expands the above for the $p$-generalized mean, i.e.\,\,for $(x_1, \ldots, x_n) \in \R^n, n\in \N$, we have
\begin{equation} \label{eq:pAGM}
{\left(\prod_{i=1}^n |x_i| \right)}^{1/n} \le \quad  {\left( \frac{1}{n} \sum_{i=1}^n |x_i|^p\right)}^{1/p}.
\end{equation}
It was shown by Gluskin and Milman \cite{GluskinMilman2003} that for a random vector $X^{(n)} \in \R^n$ uniformly distributed on the standard $(n-1)$-dimensional unit sphere $\SSS^{n-1}$ in $\R^n$, one can reverse the $p$-AGM inequality in \eqref{eq:pAGM} for $p=2$ up to a scalar constant with high probability, which was then extended to $p=1$ by Aldaz \cite{aldaz2008selfimprovemvent, aldaz2010concentration}. Kabluchko, Prochno and Vysotsky \cite{KPV} provided a central limit theorem (CLT) and a large deviation principle (LDP) for the ratio of the two sides of the $p$-AGM inequality for any $p \in [1, \infty)$ and $X^{(n)} $ uniformly distributed within the $\ell_p^n$-ball $\B_p^n$ or distributed according to the surface measure or the cone probability measure on the $\ell_p^n$-sphere $\SSS_p^{n-1}$, where 
$$\B_p^n:=\{x\in\R^n:\|x\|_p\leq 1\} \text{ \qquad and \qquad } \SSS_p^{n-1}:=\{x\in\R^n:\|x\|_p=1\},$$
with
$$
\|x\|_p := \Big(\sum\limits_{i=1}^n|x_i|^p\Big)^{1/p}.
$$
The cone probability measure on $\SSS_p^{n-1}$ is defined as 
$$\bC_{n,p}(\,\cdot\,) := {\vol_n(\{rx:r\in[0,1],x\in\,\cdot\,\})\over\vol_n(\B_p^n)},$$
where $\vol_n$ denotes the $n$-dimensional Lebesgue measure on $\R^n$. Finally, Thäle \cite{ThaeleAGM-MDP}  then expanded the results of  \cite{KPV} to a CLT and a moderate deviation principle (MDP) for the ratio of the two sides of the $p$-AGM inequality with the corresponding random vector $X^{(n)}  \in \B_p^n$ having a distribution from a wider class of $p$-radial distributions, established by Barthe, Gu\'edon, Mendelson and Naor in \cite{BartheGuedonEtAl}, which includes the uniform distribution and the cone probability measure as special cases. This class of distributions is constructed by mixing the uniform distribution and the cone probability measure via a $p$-radial density, which is given by an additional distribution on $[0,\infty)$. However, the arguments of Thäle show that the properties of interest of a random vector $X^{(n)} \in \B^n_p$ are independent of the $p$-radial component of its distribution, as long as the directional distribution is given by $\bC_{n,p}$ and its $p$-radial distribution has no atom at zero (see Section \ref{sec:ProbRep}). This means, that there is a distribution $\bR$ on $[0,1]$  with $\bR(\{0\})=0$, such that for a random variable $R$ with distribution $\bR$ and a random variable $Z^{(n)}$ with distribution $\bC_{n,p}$  that is independent of $R$, we have that 
\begin{equation}\label{eq:pRadDistr}
X^{(n)} \overset{d}{=} R \cdot Z^{(n)},
\end{equation}
where $\overset{d}{=}$ denotes equality in distribution. (This can also be expanded to sequences of $p$-radial distributions $\big(\bR^{(n)}\big)_{n_\N}$, if the limiting distribution also has no atom at zero). Thus, it follows from  \cite{ThaeleAGM-MDP} that the same CLT and MDP hold universally for the ratio of the two sides of the $p$-AGM inequality for any random vector $X^{(n)} \in \B^n_p$ with directional distribution $\bC_{n,p}$.\\

The purpose of this paper is to develop further the large deviation results of \cite{KPV} into \textit{sharp} large deviations (SLD) results in the spirit of Bahadur and Ranga Rao \cite{Bahadur}. The theory of sharp large deviations has only very recently been introduced into the field of high-dimensional convex geometry, starting with Liao and Ramanan \cite{LiaoRamanan} and followed up by Kaufmann \cite{TKSLDP}, both in the context of $\ell_p^n$-balls and -spheres. It has the distinct advantage over classical large deviations theory that it gives tail asymptotics not on a logarithmic scale and can provide concrete and asymptotically exact tail estimates for specific $n\in\N$. Moreover, just as LDP results are more sensitive to the underlying distributions than e.g.\,CLTs and MDPs, SLD results are so to an even greater extent than LDPs. For the sake of brevity, we will not recapitulate the basics of (sharp) large deviations theory here, and defer the reader to the works \cite{TKSLDP, LiaoRamanan} for an overview of the background in the current setting of $\ell_p^n$-balls, and to the classic literature \cite{Bahadur, DZ, dH, Petrov} for a more detailed account of the relevant theory.\\

For a random vector $X^{(n)}  \in \B_p^n$ with directional distribution $\bC_{n,p}$  in the sense of \eqref{eq:pRadDistr} we now want to give sharp asymptotics for the probability of the ratio of the two sides of the $p$-AGM inequality in \eqref{eq:pAGM} being bigger than a constant $\theta \in [0,1]$. To state our main result, we need to define the following functions: For $\tau=(\tau_1, \tau_2)\in \R^2$, set
\begin{equation*} 
\Lambda_p(\tau) := \log \left( \frac{1}{2p^{1/p}\Gamma\big(1+\frac{1}{p}\big)} \int_\R e^{\tau_1 \log(|y|) + \frac{1}{p}(\tau_2-1) |y|^p} \, \dint y \right),
\end{equation*}
and for  $x\in \R^2$ denote the Legendre-Fenchel transform of $\Lambda_p$ as 
$$\displaystyle \Lambda_p^*(x) := \sup_{\tau \in \R^2} \big[\langle x, \tau \rangle - \Lambda_p(\tau)\big],$$
where $\langle \, \cdot \, , \, \cdot \, \rangle$ denotes the standard scalar product in $\R^2$. We denote by $\mathcal{J}_p$ the effective domain of $\Lambda_p^*$, i.e. the set of arguments for which $\Lambda_p^*$ is finite.  For an $x \in \mathcal{J}_p$, we denote by $\tau(x) \in \R^2$ the coefficients in the above term, where the supremum is attained, i.e. where it holds that
\begin{equation*} 
\Lambda_p^*(x) =  \langle x, \tau(x) \rangle - \Lambda_p(\tau(x)).
\end{equation*}
For a detailed argument for the existence of $\tau(x)$, we refer the reader to \cite[p. 246 f.]{TKSLDP}, particularly the explanation regarding Equation (9) therein. For a function $g: \R^d \to \R^d$, we denote by $J_x g(x^*)$ the Jacobian of $g$ with respect to the vector $x$ evaluated at $x^* \in \R^d$, and for $f:\R^d \to \R$ by $\nabla_x f(x^*)$ and $\mathcal{H}_x f(x^*)$ the gradient and Hessian of $f$ with respect to the vector $x$ evaluated at $x^* \in \R^d$, respectively. Moreover, let 
$$\psi(x)= \frac{\Gamma^\prime(x)}{\Gamma(x)}, \qquad  \text{with }x>0,$$
be the digamma function. We then set 
$$m_p:= \frac{1}{p} \left(\psi \left(\frac{1}{p}\right)+ \log(p)\right) <0. $$
As we will see in Section \ref{sec:ProbRep}, $e^{m_p}$ is the limit towards which the expectations of the ratio of the $p$-AGM inequality converge in $n\in\N$. Furthermore, we need to define the functions $\xi(\theta)$ and $\kappa(\theta)$ for $\theta \in [0,1]$, as used also in the sharp large deviation results of Bahadur and Ranga Rao \cite{Bahadur}. For $x \in \R^2$, we set 
\begin{equation} \label{eq:HessX}
\displaystyle \mathfrak{H}_{x} :=  \mathcal{H}_{\tau}\Lambda_p(\tau(x))
\end{equation}
to be the Hessian of \hspace{-0.05cm}$\Lambda_p(\tau)$ in $\tau \in \R^2$, evaluated at $\tau(x)$. %
For $\theta \in (0,1]$, we denote $\theta^*:=(\log \theta,1) \in \R^2$ and for $\theta \in (0,1)$ we define
\begin{equation*} 
\displaystyle \xi(\theta)^2 := \langle \mathfrak{H}_{\theta^*} \, \tau(\theta^*), \tau(\theta^*) \rangle\, \det \mathfrak{H}_{\theta^*},
\end{equation*}
and
\begin{equation*} 
\displaystyle \kappa(\theta)^2 := 1 - c_\kappa(\theta),
\end{equation*}
with $c_\kappa(\theta)$ given by
\begin{equation*}
\frac{\left(\tau(\theta^*)_1^2 + \tau(\theta^*)_2^2\right)^{3/2} \, p^2 e^{p \theta} \theta^{-p}}{\big{|}\tau(\theta^*)_2^2 \left(\mathfrak{H}_{\theta^*}^{-1}\right)_{11} - 2 \tau(\theta^*)_1  \tau(\theta^*)_2 \left(\mathfrak{H}_{\theta^*}^{-1}\right)_{12} + \tau(\theta^*)_1^2 \left(\mathfrak{H}_{\theta^*}^{-1}\right)_{22} \big{|} \, \left(1+ p^2 e^{2 p \theta} \theta^{-2 p}\right)^{3/2}}.
\end{equation*}
\text{}\\
In the following results and throughout this paper, we denote by $o(1)$ a sequence that tends to zero as $n$ tends to $\infty$. With the necessary definitions and notation set up, we now proceed to formulate our main result.

\begin{thm}\label{thm:AGM-SLDP}
	Let $1 \le p < \infty$, $n \in \N$, and $X^{(n)}$ be a random vector in $\B^n_p$ with directional distribution $\bC_{n,p}$ in the sense of \eqref{eq:pRadDistr}. It then holds%
	\vspace{0.25cm}
	\begin{enumerate}
		\item[i)]for $\theta \in (e^{m_{p}}, 1)$ and $n$ sufficiently large that
		$$\Pro\Big[ \Big( \prod_{i=1}^n |X_i^{(n)}|\Big)^{1/n} > \theta \cdot \Big( \sum_{i=1}^n |X_i^{(n)}| \Big)^{1/p}\Big] =\displaystyle \frac{1}{\sqrt{2\pi n} \, \kappa(\theta) \xi(\theta)} \, e^{-n \, \mathcal{I}_p(\theta)} \, (1 + o(1)),$$
		\item[ii)] and for $\theta \in (0,e^{m_{p}})$ and $n$ sufficiently large that
		$$\Pro\Big[ \Big( \prod_{i=1}^n |X_i^{(n)}|\Big)^{1/n} < \theta \cdot \Big( \sum_{i=1}^n |X_i^{(n)}| \Big)^{1/p}\Big] =\displaystyle \frac{1}{\sqrt{2\pi n} \, \kappa(\theta) \xi(\theta)} \, e^{-n \, \mathcal{I}_p(\theta)} \, (1 + o(1)),$$
	\end{enumerate}
	where 
	\vspace{-0.075cm}
	\begin{eqnarray*}
		\mathcal{I}_p(\theta)&:=& [pG_p(\theta)-1]\log(\theta) 
	+ G_p(\theta)[\log G_p(\theta)-1] - \log \Gamma(G_p(\theta))\\ 
	&&+ \frac{1}{p}(1+\log(p)) + \log \Gamma\left(\frac{1}{p}\right),
	\end{eqnarray*}
	with $G_p(\theta):= H^{-1}(p\log(\theta))$, where $H:(0,\infty) \to (-\infty, 0)$ is an increasing bijection given by 
	\begin{equation}\label{eq:H}
	H(x):= \psi(x) - \log(x).
	\end{equation}
\end{thm}
The two parts of the above theorem describe the decay of the probabiliy that the $p$-AGM inequality is either reversible with a prefactor $\theta \in (e^{m_{p}}, 1)$ [part i)] or can be sharpened with a prefactor $\theta \in (0, e^{m_{p}})$ [part ii)]. Conversely, their respective opposites, i.e. the probabilities that the inequality can be reversed with a prefactor $\theta \in (0, e^{m_{p}})$ or sharpened  with a prefactor $\theta \in (e^{m_{p}}, 1)$ tend to 1 in $n \in \N$. This will be pointed out in further detail in Section \ref{sec:ProbRep}.\\

Note that the rate function $\mathcal{I}_p$ is not dependent on the $p$-radial distribution of $X^{(n)}$, as is also the case in \cite{KPV, ThaeleAGM-MDP}, even though SLD results usually tend to be more sensitive to the idiosyncrasies of the underlying distributions.\\

These results are consistent with the large deviation principle of Kabluchko, Prochno, and Vysotsky, as taking the logarithm of the probability in the above theorem, dividing by $n$, and then considering the limit, yields what they have shown in \cite[Theorem 1.2]{KPV}, namely that %
\begin{equation} \label{eq:KPV-LDP-Reverse}
\lim_{n\to\infty} \frac{1}{n} \log \Pro\Big[ \Big( \prod_{i=1}^n |X_i^{(n)}|\Big)^{1/n} > \theta \cdot \Big( \sum_{i=1}^n |X_i^{(n)}| \Big)^{1/p}\Big] = -\mathcal{I}_p(\theta)
\end{equation}
for $\theta \in (e^{m_{p}}, 1)$ and 
\begin{equation} \label{eq:KPV-LDP-Sharpen}
\lim_{n\to\infty} \frac{1}{n} \log \Pro\Big[ \Big( \prod_{i=1}^n |X_i^{(n)}|\Big)^{1/n} > \theta \cdot \Big( \sum_{i=1}^n |X_i^{(n)}| \Big)^{1/p}\Big] = -\mathcal{I}_p(\theta)
\end{equation}
for $\theta \in (0, e^{m_{p}})$. However, we do provide a refinement of their results, since Theorem \ref{thm:AGM-SLDP} gives estimates on a non-logarithmic scale and we can thereby give concrete and asymptotically exact probability estimates for the reversibility and improvability of the $p$-AGM inequality for a specific (sufficiently large) $n \in \N$, whereas the prefactor in Theorem \ref{thm:AGM-SLDP} vanishes on the logarithmic scale of a large deviation principle as in \eqref{eq:KPV-LDP-Reverse} and \eqref{eq:KPV-LDP-Sharpen}.
%
The proof of both Theorem \ref{thm:AGM-SLDP} will follow closely along the lines of Kaufmann \cite{TKSLDP} (and we defer to the proofs therein, where the arguments are analogue) and is structured in three steps, each of which will have a dedicated section. In Section \ref{sec:ProbRep} the ratio of the two sides of the $p$-AGM inequality, denoted as $\mathcal{R}_n$, will be reformulated in terms of so-called $p$-generalized Gaussian random vectors via well-established representation results of Rachev and Rüschendorf \cite{RachevRueschendorf} and Schechtmann and Zinn \cite{SchechtmanZinn}. Furthermore, the large deviation results of Kabluchko, Prochno, and Vysotsky  \cite{KPV} for $\mathcal{R}_n$ will be given explicitly and expanded to general distributions with directional component $\bC_{n,p}$. In Section \ref{sec:JointDensityEstimate} we will provide a local density approximation for this probabilistic representation and then prove the main result in Section \ref{sec:ProofMainResult} by integrating over the density estimate. For the latter, a geometric result for Laplace integral expansions from Adriani and Baldi \cite{AdrianiBaldi} is utilized.

\section{Probabilistic Representation}\label{sec:ProbRep}

For a random vector $X^{(n)} \in \B^n_p$ with directional distribution $\bC_{n,p}$ in the sense of \eqref{eq:pRadDistr} the main variable of interest is the ratio of the two sides of the $p$-AGM inequality given as
\begin{equation} \label{eq:DefRn}
\mathcal{R}_n := \frac{{\left(\prod_{i=1}^n |X_i^{(n)} | \right)}^{1/n}}{{\left( \frac{1}{n} \sum_{i=1}^n |X_i^{(n)} |^p\right)}^{1/p}}.
\end{equation}
We want to formulate the target probabilities $\Pro(\mathcal{R}_n > \theta)$ and $\Pro(\mathcal{R}_n < \theta)$ via a random vector $Y^{(n)}$ with generalized Gaussian distribution. In general, we say a real-valued random variable $X$ has a generalized Gaussian distribution if its distribution has Lebesgue density 
$$\displaystyle f_{\textup{gen}}(x):=  \displaystyle \frac{b}{2 \, a\, \Gamma\left(\frac{1}{b}\right)} \, e^{-\big({|x- m|}/{a}\big)^b},\qquad x\in\R,$$
where $m \in \R$ and $a,b>0$, and denote this by $X \sim {\bN}_{\textup{gen}}(m, a, b)$. For our probabilistic representation, we will specifically use the generalized Gaussian distribution $\bN_p := {\bN}_{\textup{gen}} \left(0, p^{1/p}, p \right)$, $p \in [1, \infty),$ often referred to as a $p$-generalized Gaussian distribution, with density
$$
\displaystyle f_p(x) := \frac{1}{2 \, p^{1/p} \, \Gamma\big(1+\frac{1}{p}\big)}\, e^{-{|x|^p}/{p}},  \qquad x\in\R.
$$
With this, we have the following useful representation resultfor the cone probability measure $\bC_{n,p}$ shown in \cite{RachevRueschendorf} and \cite{SchechtmanZinn}. 
\begin{proposition}\label{prop:ProbRepCnp}
Let $1 \le p<\infty$, $Y=(Y_1,\ldots,Y_n)$ be a random vector with $Y_i \sim \bN_p$ i.i.d., then
 the random vector $  {Y}/{\|Y\|_p}$ has distribution $\bC_{n,p}$ and is independent of $\|Y\|_p,$.
\end{proposition}
\begin{rmk}\label{prop:ExpRadialDistr}
Consider a random vector $X^{(n)} \in \B^n_p$ with directional distribution $\bC_{n,p}$ and $p$-radial distribution $\bR$ on $[0,1]$ in the sense of \eqref{eq:pRadDistr}. If $\bR$ is the Dirac measure at $1$, the overall distribution of $X^{(n)}$ is again the cone measure $\bC_{n,p}$. Choosing $\bR$ to be a beta distribution $\bB(\frac{n}{p},1)$ causes $X^{(n)}$ to be uniformly distributed in $\B^n_p$. Finally, for $m\in \N$, setting $\bR$ to be a beta distribution $\bB(\frac{n}{p},\frac{m}{p})$, the distribution of $X^{(n)}$ then corresponds to the orthogonal projection of $\bC_{n+m,p}$ on $\B^{n+m}_p$ onto its first $n$ coordinates, which for $p=m$ interestingly yields the uniform distribution on $\B^n_p$. All of these identities follow from \cite[Theorem 1 \& 2, Corollary 3]{BartheGuedonEtAl} by calculating the $p$-radial distributions of the random variables considered therein (see also \cite[Section 3]{PTTSurvey}).
\end{rmk}
It directly follows from Proposition \ref{prop:ProbRepCnp} for a random vector $X^{(n)} \in \B^n_p$ with directional distribution $\bC_{n,p}$ and $p$-radial distribution $\bR$ on $[0,1]$ in the sense of \eqref{eq:pRadDistr} that
\begin{equation} \label{eq:ProbRepRn}
\displaystyle \mathcal{R}_n \overset{d}{=} \frac{{\left(\prod\limits_{i=1}^n \left|	R \,  \frac{Y_i^{(n)}}{\|Y_i^{(n)}\|_p} \right| \right)}^{1/n}}{{\left( \frac{1}{n} \sum\limits_{i=1}^n \left|	R \,  \frac{Y_i^{(n)}}{\|Y_i^{(n)}\|_p}  \right|^p\right)}^{1/p}}
= \frac{{\left(\prod\limits_{i=1}^n \left|	Y_i^{(n)} \right| \right)}^{1/n}}{{\left( \frac{1}{n} \sum\limits_{i=1}^n \left|	Y_i^{(n)} \right|^p\right)}^{1/p}}.
\end{equation}
Thus, we see that $\mathcal{R}_n$ does not depend on the $p$-radial distribution $\bR$, which is why the rate function in the main result is universal for all random vectors in $\B^n_p$ with directional distribution $\bC_{n,p}$. This calculation also shows that the CLT and LDP established in \cite{KPV} and the MDP shown by Thäle \cite{ThaeleAGM-MDP} also hold for any random vector in $\B^n_p$ with directional distribution $\bC_{n,p}$ in the sense of \eqref{eq:pRadDistr}. In the light of the above argument in \eqref{eq:ProbRepRn}, let us present the LDP based on \cite[Theorem 1.4]{KPV} here in this more general form. 
\begin{proposition} \label{prop:LDP-Rn}
	Let $1 \le p < \infty$ and $X^{(n)}$ be a random vector in $\B_p^n$ with directional distribution $\bC_{n,p}$ in the sense of \eqref{eq:pRadDistr}. Then the sequence $(\mathcal{R}_n)_{n \in \N}$ with $\mathcal{R}_n$ as defined in \eqref{eq:DefRn} based on $X^{(n)}$ satisfies an LDP on $[0,1]$ with speed $n$ and rate function $\mathcal{I}_p$ as in Theorem \ref{thm:AGM-SLDP}.
\end{proposition}
It is furthermore shown in \cite{KPV} that $\mathcal{I}_p(e^{m_p})=0$ and $\mathcal{I}_p(0+) = \mathcal{I}_p(1-) = +\infty$, where $\mathcal{I}_p(0+)$ and $\mathcal{I}_p(1-)$ denote the limits of $\mathcal{I}_p$ to for sequences that converge to $0$ and $1$ from above and below, respectively. As suggested by the central limit result in \cite[Theorem 1.1]{KPV}, the expectations of $\mathcal{R}_n$ converge to $e^{m_p}$, i.e. the value from which deviation probabilities are given in the above LDP and by the SLD results in this paper.\\

For the $p$-AGM inequality this means two things: On the one hand, for sufficiently large $n\in \N$, it can be reversed with prefactors in $(0, e^{m_p})$ with high probability, and for prefactors in $(e^{m_p},1)$ the decay of this ``reversion probability'' is described by the rate function in Theorem \ref{thm:AGM-SLDP} i). On the other hand, for sufficiently large $n\in \N$, it can be improved with prefactors in $(e^{m_p},1)$ with high probability, with the decay of this ``improvement probability'' for prefactors in $(0, e^{m_p})$ being described by the rate function in Theorem \ref{thm:AGM-SLDP} ii).\\ 

Proposition \ref{prop:LDP-Rn} is proven in \cite{KPV} by showing an LDP for the sequence of empirical averages of the coordinates of the random vector
\begin{equation} \label{eq:DefVi}
V^{(n)} := \left(V_1^{(n)}, \ldots, V_n^{(n)}\right), \qquad \text{with} \qquad  V_i^{(n)} := \left( \log |Y_i^{(n)}|,  |Y_i^{(n)}|^p \right),
\end{equation}
with $Y_i^{(n)}$ i.i.d. and $Y_i^{(n)} \sim \bN_p$. This is done via Cram\'er's theorem (see e.g. \cite[Theorem 2.2.30, Theorem 6.1.3, Corollary 6.1.6]{DZ}), i.e. by showing that the cumulant generating function $\Lambda_{V}(\tau):= \E e^{\langle \tau, V_i^{(n)}\rangle}, \tau \in \R^2$ of the $V_i^{(n)}$ is finite in a neighbourhood of the origin, hence the sequence of empirical averages of the coordinates
\begin{equation} \label{eq:DefSn}
S^{(n)} := \frac{1}{n} \sum \limits_{i=1}^n V_i^{(n)} = \frac{1}{n} \sum \limits_{i=1}^n \left( \log |Y_i^{(n)}|,  |Y_i^{(n)}|^p \right) 
\end{equation}
satisfies an LDP with speed $n$ and rate function $\Lambda_{V}^*$, which is the Legendre-Fenchel transform of $\Lambda_{V}$. Note that $\Lambda_{V}= \Lambda_p$ and $\Lambda_{V}^*= \Lambda_p^*$. This LDP is then mapped to the sequence $(\mathcal{R}_n)_{n \in \N}$ via the representation result from Proposition \ref{prop:ProbRepCnp} and the contraction principle (see e.g. \cite[Theorem 4.2.1]{DZ}), considering the map $F(x_1, x_2) := e^{x_1}{x_2}^{-1/p}$, yielding an LDP for $(\mathcal{R}_n)_{n \in \N}$ with speed $n$ and rate function  
$$\inf\limits_{(x_1, x_2): \,  F(x_1, x_2)=\theta} \, \Lambda_p^*(x_1, x_2), \qquad  \theta \in [0,1].$$
This is then finalized by showing that the above infimum is attained uniquely at $\theta^* :=(\log \theta, 1)$ and that this infimum can be given explicitly as
\begin{equation} \label{eq:EindInfLambda*}
\inf\limits_{(x_1, x_2): \,  F(x_1, x_2)=\theta} \, \Lambda_p^*(x_1, x_2) =  \Lambda_p^*(\theta^*) = \mathcal{I}_p(\theta). 
\end{equation}
Further, it is shown that the effective domain of $\mathcal{I}_p$ is $(0,1)$ and that for $x \in \mathcal{J}_p$ it holds that 
\begin{equation} \label{eq:ValuesTau}
\tau(x)= \Big( p \,H^{-1}(px_1 - \log x_2)-1,  \frac{1}{p} - x_2^{-1}\,H^{-1}(px_1 - \log x_2)\Big)
\end{equation}
with $H$ as in \eqref{eq:H} (see \cite[p. 11 f.]{KPV}). We will use the same probabilistic representations from \eqref{eq:DefVi} and \eqref{eq:DefSn}, but proceed with them in a different fashion. It holds that 
\begin{equation} \label{eq:ReformTargetProbs}
\Pro(\mathcal{R}_n > \theta) = \Pro(S^{(n)} \in D_{\theta, >}) \quad \text{and} \quad \Pro(\mathcal{R}_n < \theta) = \Pro(S^{(n)} \in D_{\theta, <}),
\end{equation}
with 
\begin{eqnarray} \label{eq:DefD_theta>}
D_{\theta, >} := \{x \in \R^2: x_2 > 0, e^{x_1}{x_2}^{-1/p} > \theta\}, 
\end{eqnarray}
and 
\begin{eqnarray} \label{eq:DefD_theta<}
D_{\theta, <} := \{x \in \R^2: x_2 > 0, e^{x_1}{x_2}^{-1/p} < \theta\}.
\end{eqnarray}
%
\begin{rmk} \label{rmk:UniquenessInf}
	Note, that the points satisfying the infimum condition $F(x_1, x_2) = \theta$ in \eqref{eq:EindInfLambda*} are exactly those on the boundary $\{(x_1, x_2) \in \R^2: x_2 > 0, e^{x_1}{x_2}^{-1/p} = \theta \}$ of $D_{\theta, >}$ and $D_{\theta, >}$ (which coincide). Hence, \eqref{eq:EindInfLambda*} shows that the infimum of $\Lambda_p^*$ over this boundary is uniquely attained at $\theta^*$. 
\end{rmk}
As mentioned in the introduction, we will proceed to give an asymptotic density approximation $h^{(n)}$ for $S^{(n)}$ in the following section, such that for sufficiently large $n \in \N$ we can write the probabilities in \eqref{eq:ReformTargetProbs} as integrals of $h^{(n)}$ over $D_{\theta, >}$ and $D_{\theta, <}$. The integral approximation result by Adriani and Baldi \cite{AdrianiBaldi}, which we use to calculate that integral explicitly in the final section, will then show that the values of the above integrals are heavily dependent of the geometric form of the boundaries of $D_{\theta, >}$ and $D_{\theta, <}$. Thus, the direct influence of the mapping $F(x_1, x_2) := e^{x_1}{x_2}^{-1/p}$ on the LDP through the contraction principle will have a more subtle form in the SLD setting, as it influences the central results via the geometric properties of its graph, seen as the local boundary of the sets $D_{\theta, >}$ and $D_{\theta, <}$.%
%
%
%
\section{Density Approximation}\label{sec:JointDensityEstimate}
The sought-after density approximation for $S^{(n)} = \frac{1}{n} \sum _{i=1}^n ( \log |Y_i^{(n)}|,  |Y_i^{(n)}|^p)$ can be derived by the original result of Borovkov and Rogozin \cite[Theorem 1]{BorovkovRogozin}, however we use a more convenient version of it from \cite[Theorem 3.1]{AdrianiBaldi}. For a sum of i.i.d.\,\,random vectors with bounded common density it provides a local density estimate for their empirical averages. We will not delve too deep into the inner workings of the proof, however we will outline its core idea, so as to argue why this result is still applicable to $S^{(n)}$, even though the random vectors $( \log |Y_i^{(n)}|,  |Y_i^{(n)}|^p)$ clearly have no common bounded density in $\R^2$, as their coordinates are highly dependent.  We start off by stating the result itself in the two-dimensional setting, using the identities in \cite[Equation (2.3)]{AdrianiBaldi}.
\begin{proposition}\label{prop:DensityEstimateAdrianiBaldi}
	Let $(X_n)_{n\in\N}$ be a sequence of $\R^2$-valued random vectors. Assume that their common law $\mu$ has a bounded density with respect to the Lebesgue measure and that their Laplace transforms are finite in a neighbourhood of the origin. Then, for every  for every $x$ in the effective domain of the Legendre-Fenchel transform of the cumulant generating function $\Lambda$, $S^{(n)} := \frac{1}{n} \sum_{i=1}^n  X_n$ has a density $h^{(n)}$ with
	$$ h^{(n)}(x) = {\frac{n}{2\pi}} \, {(\det \mathfrak{H}_{x})}^{-1/2} \,  e^{- n \, \Lambda^*(x)} (1+o(1)),$$
	where $\mathfrak{H}_{x}:=  \mathcal{H}_{\tau}\Lambda_p(\tau(x))$ as in \eqref{eq:HessX}. %
\end{proposition}
This is proven using the so called saddle point method (see e.g.\,\,Jensen \cite{JensenBook} for details), which provides a local density of the empirical average via the Fourier transform of the involved i.i.d.\,\,random vectors, which are often easier to obtain than their actual densities. Using the Fourier inversion theorem, the (unknown) density of the empirical average is written as a complex integral of the Fourier transform. Cauchy's theorem then makes is possible to transform the path of integration in such a way that it passes through a complex saddle point, such that the mass of the integral heavily concentrates in $n \in \N$ around this saddle point. Then, well-established integral approximation techniques can be used with high accuracy, to write the density explicitly.\\

The above process of course requires the involved Fourier transform to be integrable in order to apply the Fourier inversion theorem, which is ensured in \cite{AdrianiBaldi} by the i.i.d.\,\,random vectors having a bounded common density. It is, however, remarked in \cite[Remark 3.2]{AdrianiBaldi}, that any other argument to the same effect could be applied as well. In the context of this paper, the Fourier transform of the random vectors $V^{(n)}_i :=(\log |Y_i^{(n)}|, |Y_i^{(n)}|^p)$ can be obtained via the underlying distribution $\bN_p$ of the random variables $Y_i^{(n)}$. 
One can then deduce the integrability of the Fourier transform of their empirical average for sufficiently large $n\in \N$, i.e.\,\,for $n>n_0$, which is dependent on $x$ and $p$, by using the properties of the density $f_p$ of $\bN_p$ and applying the Hausdorff-Young inequality. Since this was done in detail by Liao and Ramanan in \cite[Lemma 6.1]{LiaoRamanan} for a sequence of random vectors that is very similar to ours, the proof will be completely analogue and we refrain from reiterating it here. Thus, despite the lack of a bounded density for the $V^{(n)}_i$, Proposition \ref{prop:DensityEstimateAdrianiBaldi} can still be applied (with the added condition of $n \in \N$ being sufficiently large), and yields the following proposition. 

\begin{proposition}\label{prop:DensityEstimateSn}
Let $p \in [1, \infty)$ and $n\in \N$. For $S^{(n)} = \frac{1}{n} \sum_{i=1}^n  V^{(n)}_i$ with $V^{(n)}_i = (\log{|Y^{(n)}_i|}, |Y^{(n)}_i|^p),$ $\, Y_i^{(n)} \sim \bN_p$ i.i.d., $x \in \mathcal{J}_p$, and $n$ sufficiently large, it holds that the distribution of $S^{(n)} $ has Lebesgue density
	$$ h^{(n)}(x) = \frac{n}{2\pi} \, {(\det \mathfrak{H}_{x})}^{-1/2} \,  e^{- n \, \Lambda_p^*(x)} \, (1 + o(1)),$$
	with $\mathfrak{H}_{x}$ as in \eqref{eq:HessX}.%
\end{proposition}
\section{Proof of Main Result }\label{sec:ProofMainResult}
Assuming the set-up of Theorem \ref{thm:AGM-SLDP} and combining the probabilistic representation results in \eqref{eq:ProbRepRn} and \eqref{eq:ReformTargetProbs} with the local density approximation in Proposition \ref{prop:DensityEstimateSn}, we get that 
\begin{eqnarray} \label{eq:ReformTargetProbAsDensityInt>}
\nonumber \Pro(\mathcal{R}_n > \theta) = \Pro(S^{(n)} \in D_{\theta, >})&=& \int_{D_{\theta, >}} h^{(n)}(x) \, \dint x\\
&=& \frac{n}{2\pi} \, \int_{D_{\theta, >}} {(\det \mathfrak{H}_{x})}^{-1/2} \,  e^{- n \, \Lambda_p^*(x)} \,  \dint x \, (1 + o(1)),
\end{eqnarray}
and 
\begin{eqnarray} \label{eq:ReformTargetProbAsDensityInt<}
\nonumber\Pro(\mathcal{R}_n < \theta) = \Pro(S^{(n)} \in D_{\theta, <})&=& \int_{D_{\theta, >}} h^{(n)}(x) \,\dint x\\
&=& \frac{n}{2\pi} \, \int_{D_{\theta, <}} {(\det \mathfrak{H}_{x})}^{-1/2} \,  e^{- n \, \Lambda_p^*(x)} \, \dint x \, (1 + o(1)),
\end{eqnarray}
with $D_{\theta, >}$ and $D_{\theta, <}$ as in \eqref{eq:DefD_theta>} and \eqref{eq:DefD_theta<}. The final step of the proof of our main result now is to calculate the above integrals explicitly. We will only do this in detail for the integral in \eqref{eq:ReformTargetProbAsDensityInt>}, as the calculation for the integral in \eqref{eq:ReformTargetProbAsDensityInt<} proceeds in a mostly analogue fashion, and we will merely point out the specific differences at the end of the proof. As in \cite{AdrianiBaldi, TKSLDP, LiaoRamanan}, the first step is to split up the integration area into a neighbourhood around the point $\theta^*$, at which the exponent in the integrand attains its infimum on the boundary of $D_{\theta, >}$, and its complement. On this neighbourhood we then employ a geometric result for Laplace integral approximations by Adriani and Baldi \cite{AdrianiBaldi}, and on the complement we use the large deviation principle from Proposition \ref{prop:LDP-Rn} to show the comparative negligibility of the corresponding integral. The result for Laplace integral approximations is geometric in the sense that it approximates the integral using the Weingarten maps of the $\theta^*$-level set of $\Lambda_p^*$ and the boundary of $D_{\theta, >}$ within the chosen neighbourhood, both seen as planar curves.\\

For a brief recapitulation of the Weingarten map, we refer to \cite[Section 4]{AdrianiBaldi} or \cite[Section 2.5]{TKSLDP}, and to \cite{Hicks, Klingenberg} for a more comprehensive resource. We will simply note that in two-dimensional space, the Weingarten map of a curve at a given point is simply the absolute value of its curvature at this point. Hence, we introduce the following derivative notation and recall formulas for curvatures of two types of planar curves: For a map $f:\R^d \to \R$ and some $x^* \in \R^d, d\in \N,$ we use the multi-index notation
\begin{equation} \label{eq:AblNotation}
f_{[i_1, \ldots, i_d]}(x^*) := \displaystyle \frac{\partial^{i_1}}{\partial x_1^{i_1}} \ldots \frac{\partial^{i_d}}{\partial x_d^{i_d}} \, f (x)\Big{|}_{x=x^*},
\end{equation}
with $i_1, \ldots, i_d \in \N$. The following lemma provides formulas for the curvature of planar curves, specifically for implicit curves, that is, curves given as the zero set of a function, and for curves that are the graph of a function. Both follow from the curvature formula given by Goldman in \cite[Proposition 3.1]{Goldman2005}.
\begin{lemma} \label{lem:Curvature}\text{}
	\begin{itemize}%
		\item[i)]{Let $F: \R^2 \to \R$ be a twice differentiable function. For a curve $\mathscr{C}:= \{x \in \R^2: F(x) =0\}$ given as the zero set of $F$, and a point $p \in \mathscr{C}$, where $\nabla_{x} F(p) \ne 0$, it holds for the curvature $K$ of $\mathscr{C}$ in $p$ that
			$$K(p) = \displaystyle \frac{{F_{[0,1]}}^2{F_{[2,0]}} - 2{F_{[0,1]}}{F_{[1,0]}}{F_{[1,1]}} + {F_{[1,0]}}^2{F_{[0,2]}}}{{\left({F_{[1,0]}}^2 + {F_{[0,1]}}^2\right)}^{3/2}},$$ }
		with $F_{[i,j]} = F_{[i,j]}(p)$ as in \eqref{eq:AblNotation}.
		\item[ii)]{In case that $\mathscr{C}$ is the graph of a twice differentiable function $f:\R \to \R$, i.e. $\mathscr{C} = \{(x_1, x_2) \in \R^2: x_2 = f(x_1)\}$, and $p=(x, f(x))$, the above reduces to 
			$$K(p) = \displaystyle \frac{f^{\prime\prime}(x)}{{\big(1 + f^\prime(x)^2\big)}^{3/2}}.$$}
	\end{itemize}
\end{lemma}
The last result we will present before the proof of our main result is the aforementioned Laplace integral approximation via the Weingarten map of Adriani and Baldi \cite{AdrianiBaldi}. The following proposition is the reduction of \cite[Section 4]{AdrianiBaldi} into a singular concise result, as in \cite[Lemma 5.6]{LiaoRamanan}, for $d=2$. 
Note, that for a set $D\subset \R^2$, we write $\partial D, \overline{D}, D^\circ,$ and $D^c$ for its boundary, closure, interior and complement, respectively. 
\begin{proposition} \label{prop:AdrianiBaldiWeingarten} Let $D\subset \R^2$ be a bounded domain such that $\partial D$ is a differentiable planar curve in $\R^2$. Furthermore, let $g:\R^2 \to \R$ be a differentiable function and $\phi: D \to [0, \infty)$ a nonnegative function that is twice differentiable and attains a unique infimum over $\overline{D}$ at $x^* \in \partial D$. Define the curves %
	$$\mathscr{C}_{D}=\partial D  \qquad \text{ and } \qquad \mathscr{C}_\phi=\{x \in \R^2: \phi(x) = \phi(x^*)\},$$
	and denote by $L_{D}$ and $L_\phi$ their respective Weingarten maps at $x^*$. Then it holds that 
	\begin{eqnarray*}
		&&\int_{D} g(x) \, e^{-n\, \phi(x)} \, \dint x\\
		 &&= \frac{{(2\pi)}^{1/2} \, \det(L_\phi^{-1}(L_\phi - L_{D}))^{-1/2}}{n^{3/2} \, \langle {\mathcal{H}_x\,\phi(x^*)}^{-1} \, \nabla_x \phi (x^*), \nabla_x \phi (x^*) \rangle^{1/2}} \, g(x^*) \, e^{-n \, \phi(x^*)}(1 + o(1)).
		\end{eqnarray*}
\end{proposition}
At this point, we have gathered the appropriate tools we need to proceed with proving our main result. 
\begin{proof}[Proof of Theorem \ref{thm:AGM-SLDP}] 
	We begin by proving the statement in Theorem \ref{thm:AGM-SLDP} \textup{i)}. Let us assume the setting therein and let $B_\theta$ be an open neighbourhood of $\theta^*$, small enough such that $B_\theta \subset \mathcal{J}_p$. The fact that $\theta^* \in \mathcal{J}_p$ follows from the fact that $\Lambda_p^*(\theta^*)= \mathcal{I}_p(\theta)< \infty$ for $\theta \in (0,1)$, as seen in Proposition \ref{prop:LDP-Rn}. Splitting up the reformulation of our target probability in \eqref{eq:ReformTargetProbAsDensityInt>} into integrals of $h^{(n)}$ over $B_\theta$ and $B_\theta^c$ yields
	\begin{eqnarray} \label{eq:SplitIntegral}
	\Pro(\mathcal{R}_n > \theta) &=& \int_{D_{\theta, >} \cap B_\theta} h^{(n)}(x) \, \dint x + \int_{D_{\theta, >} \cap B^c_\theta} h^{(n)}(x) \, \dint x.
	\end{eqnarray}
	We begin by showing the comparative negligibility of the second integral term. We know from Remark \ref{rmk:UniquenessInf} that $\Lambda_p^*$ attains its unique infimum on $\partial D_{\theta, >}$ at $\theta^*$. This property can be shown to hold for the closure $\overline{D}_{\theta, >}$ as follows: assume $t\in\R^2$ with $t \in D_{\theta, >}^\circ$, i.e. $e^{t_1}t_2^{-1/p}>\theta$. We then consider $\tilde \theta := e^{t_1}t_2^{-1/p}$. If $\tilde \theta^* \notin \mathcal{J}_p$, it trivially holds that $\Lambda_p^*(\theta^*) < \Lambda_p^*(t) = \infty$. Hence, assume that $\tilde \theta^* \in \mathcal{J}_p$. It now follows that $t \in \partial D_{\tilde \theta, >}$, which yields that $\Lambda_p^*(t) > \Lambda_p^*(\tilde \theta^*) = \mathcal{I}_p(\tilde \theta)$ by Remark \ref{rmk:UniquenessInf}. By the same arguments as in \cite[p. 375]{AdrianiBaldi}, we know that the Hessian of $\Lambda_p^*$ is strictly positive definite on its effective domain $\mathcal{J}_p$, and therefore $\Lambda_p^*$ is strictly convex on $\mathcal{J}_p$. From Proposition \ref{prop:LDP-Rn} we have that $\Lambda_p^*({e^{m_p}}^*) = \mathcal{I}_p(e^{m_p}) = 0$, thus $\mathcal{I}_p$ is strictly increasing on $(e^{m_p}, 1)$, thus for $\tilde \theta > \theta$ we have $\Lambda_p^*(t) >  \mathcal{I}_p(\tilde \theta) > \mathcal{I}_p(\theta) = \Lambda_p^*(\theta^*)$, thereby proving that $\Lambda_p^*$ attains its unique infimum on $\overline{D}_{\theta, >}$ at $\theta^*$. Therefore, it follows from $\theta^* \notin B_\theta^c$ that there is an $\eta >0$, such that
	$$ \inf_{t \in D_{\theta} \cap \, B_\theta^c} \Lambda_p^*(t) >\Lambda_p^*(\theta^*) + \eta.$$
	The LDP in Proposition \ref{prop:LDP-Rn} then implies that
	$$ \limsup_{n \to \infty} \frac{1}{n} \log \Pro(S^{(n)} \in D_{\theta} \cap \,B_\theta^c) \le - \inf_{y \, \in \,  D_{\theta} \cap \, B_\theta^c} \Lambda_p^*(y) \le - \Lambda_p^*(\theta^*) - \eta,$$
	from which it follows that
	\begin{equation} \label{eq:PropRepIntCompl}
	\Pro \left(S^{(n)} \in D_{\theta} \cap \, B_\theta^c \right) \le e^{-n \, \Lambda_p^*(\theta^*) - n \, \eta} \,  (1+ o(1)) = \frac{1}{e^{n \,\eta}} \, e^{-n \Lambda_p^*(\theta^*)} (1+ o(1)) .
	\end{equation} 
	Due to the leading exponential term $e^{-n \,\eta}$, the above will be comparatively negligible compared to the other integral term %
	\begin{eqnarray} \label{eq:DensityIntegralNeighbourhoodPreAdreaniBaldi}
	\nonumber &&\int_{D_{\theta, >} \cap B_\theta} h^{(n)}(x) \, \dint x\\
	&& =  \frac{n}{2\pi} \, \int_{D_{\theta, >} \cap B_\theta} {(\det \mathfrak{H}_{x})}^{-1/2} \,  e^{- n \, \Lambda_p^*(x)} \, \dint x \, (1 + o(1)),
	\end{eqnarray}
	which we will concretely calculate in the following. The clear course of action for this will be to apply Proposition \ref{prop:AdrianiBaldiWeingarten} to the integral in \eqref{eq:DensityIntegralNeighbourhoodPreAdreaniBaldi} with $D = D_{\theta, >} \cap \, B_\theta \subset \R^2$, $x^* = \theta^*$, $g(x):= {(\det \mathfrak{H}_{x})}^{-1/2} $ and $\phi(x) = \Lambda_p^*(x)$. However, we first need to check whether the conditions of Proposition \ref{prop:AdrianiBaldiWeingarten} indeed hold. The area of integration is clearly bounded and since for sufficiently small $B_\theta$ it follows from \eqref{eq:DefD_theta>} that $\partial(D_{\theta, >} \cap B_\theta)$ around $\theta^*$ is a section of the graph of the differentiable function $f(t_1)= \theta^{-p}e^{pt_1}$, it is indeed a differentiable planar curve. For the twofold differentiability of $\Lambda_p^*$ we refer to the argument in the proof of the main result in \cite[p.\,\,259]{TKSLDP} based on properties of the moment generating function and the Legendre-Fenchel transform and the implicit function theorem, as it can be applied in the same fashion here to show infinite differentiability of $\Lambda_p^*$. This, in turn, also yields the infinite differentiability of $ \mathfrak{H}_{x} =  \mathcal{H}_{\tau}\Lambda_p(\tau(x))$, and hence $g(x):= {(\det \mathfrak{H}_{x})}^{-1/2} $ is differentiable. As $\Lambda_p^*$ is a rate function for the LDP satisfied by $S^{(n)}$, it follows by the standard properties of rate functions that it is non-negative. Finally, as was shown in \cite{KPV}, it attains a unique infimum on $\partial(D_{\theta, >} \cap B_\theta)$ in $\theta^*$ (see \eqref{eq:EindInfLambda*} and Remark \ref{rmk:UniquenessInf}), which also holds for the entirety of $\overline{D_{\theta, >} \cap B_\theta}$, as was shown above. Hence, we can apply Proposition \ref{prop:AdrianiBaldiWeingarten} as intended, which gives
	\begin{eqnarray}\label{eq:IntPostWeingartenAMG1}
	\nonumber && \displaystyle  \int_{D_{\theta, >} \cap B_\theta} h^{(n)}(x) \, \dint x \\
	&& =   \displaystyle \frac{1}{\sqrt{2\pi n}} \, \frac{ \det(L_\Lambda^{-1}(L_\Lambda - L_{D}))^{-1/2} \, e^{-n \, \Lambda_p^*(\theta^*)}}{ \langle {\mathcal{H}_x\,\Lambda_p^* (\theta^*)}^{-1} \, \nabla_x\,\Lambda_p^*(\theta^*), \nabla_x\,\Lambda_p^*(\theta^*) \rangle^{1/2} \,  (\det \mathfrak{H}_{\theta^*})^{1/2}} \, (1 + o(1)), 
	\end{eqnarray}
	where $L_\Lambda$ and $L_{D}$ are the respective Weingarten maps of the curves
	$$\mathscr{C}_{D}=\partial (D_{\theta, >} \cap \, B_\theta)  \qquad \text{ and } \qquad \mathscr{C}_{\Lambda}=\{x \in \R^2: \Lambda_p^*(x) = \Lambda_p^*(\theta^*)\}$$
	at $\theta^*$. We now need to resolve the different components in this fraction. It was shown in \cite[Lemma 21]{TKSLDP} that $\nabla_{x} \Lambda_p^*(x) = \tau(x),$ and  $\mathcal{H}_{x}  \Lambda_p^*(x) = \mathfrak{H}_{x}^{-1}$, which holds in our setting by the very same arguments as presented therein. This allows rewriting the term in the denominator in \eqref{eq:IntPostWeingartenAMG1} as %
	\begin{eqnarray} \label{eq:XiFormel} 
	\nonumber \qquad \qquad &&\Big\langle {\mathcal{H}_x\,\Lambda_p^*(\theta^*)}^{-1} \, \nabla_x\,\Lambda_p^*(\theta^*), \nabla_x\,\Lambda_p^*(\theta^*) \Big\rangle \det \mathfrak{H}_{\theta^*} \qquad \qquad \qquad \qquad\\
	  \qquad \qquad&&\qquad \qquad \qquad \quad= \Big\langle \mathfrak{H}_{\theta^*} \, \tau(\theta^*), \tau(\theta^*) \Big\rangle \det \mathfrak{H}_{\theta^*} = \xi(\theta)^2.
	\end{eqnarray}
	In the following we shall give the Weingarten maps of the curves $\mathscr{C}_{D}$ and $\mathscr{C}_{\Lambda}$ explicitly and see that $\det(L_\Lambda^{-1}(L_\Lambda - L_{D})) = 1- {L_D}/{L_\Lambda} = \kappa(\theta)^2$ (the determinant falling away due to the Weingarten maps being one-dimensional). As discussed at the beginning of the section, the Weingarten map of a planar curve at a given point reduces to the absolute value of its curvature at that point, for which we have given concrete formulas in Lemma \ref{lem:Curvature}. As $\mathscr{C}_{D}=\partial (D_{\theta, >} \cap \, B_\theta)$ around $\theta^*$ is a segment of the graph of $f(t_1)= \theta^{-p}e^{pt_1}$, we get from Lemma \ref{lem:Curvature} ii) that 
	\begin{equation}\label{eq:WeingartenD}
	L_D = \frac{|f^{\prime \prime} (\theta)|}{(1 + f^{\prime}(\theta)^2)^{3/2}} = \frac{p^2 e^{p \theta} \theta^{-p}}{\left(1+ p^2 e^{2 p \theta} \theta^{-2 p}\right)^{3/2}}.
	\end{equation}
	The curve $\mathscr{C}_{\Lambda}$ can be written as the zero set of the function $F(x) := \Lambda_p^*(x) - \Lambda_p^*(\theta^*)$, and its derivatives $F_{[i,j]}$ at $\theta^*$ as in Lemma \ref{lem:Curvature} i) are known from the identities $\nabla_{x} \Lambda_p^*(x) = \tau(x)$ and  $\mathcal{H}_{x}  \Lambda_p^*(x) = {\mathfrak{H}_{x}}^{-1}$ from \cite[Lemma 21]{TKSLDP}. (Note, that for $\theta=e^{m_p}$ we have that $\mathscr{C}_{\Lambda}$ is the zero set of $F(x) = \Lambda_p^*(x)$, since $\Lambda_p^*({e^{m_p}}^*) =0$. By \eqref{eq:ValuesTau} it follows that $\tau(x)=0$ only if $x= {e^{m_p}}^*$. Hence, the zero set of $F(x)=\Lambda_p^*(x)$ is solely ${e^{m_p}}^*$, which is not a differentiable curve, and hence is not accessible by these geometric methods). It thus follows that
	\begin{equation*}
	L_\Lambda =  \displaystyle \frac{\left |\tau(\theta^*)_2^2 \left(\mathfrak{H}_{\theta^*}^{-1}\right)_{11} - 2\tau(\theta^*)_1 \tau(\theta^*)_2 \left(\mathfrak{H}_{\theta^*}^{-1}\right)_{12} + \tau(\theta^*)_1^2 \left(\mathfrak{H}_{\theta^*}^{-1}\right)_{22}\right| }{{\big(\tau(\theta^*)_1^2 + \tau(\theta^*)_2^2\big)}^{3/2}}.
	\end{equation*}
	This, together with \eqref{eq:WeingartenD}, now yields that $1- {L_D}/{L_\Lambda} = \kappa(\theta)^2$, which combined with \eqref{eq:EindInfLambda*} and \eqref{eq:XiFormel} gives 
	\begin{eqnarray}\label{eq:PropRepIntWithin}
	\displaystyle \int_{D_{\theta, >} \cap \, B_\theta} h^{(n)}(x) \, \dint x &= \displaystyle \frac{1}{\sqrt{2\pi n} \, \xi(\theta) \, \kappa(\theta)}   e^{-n \mathcal{I}_p(\theta)} \, (1 + o(1)).
	\end{eqnarray}
	Comparing  \eqref{eq:PropRepIntWithin} with the upper bound of the integral outside of $B_\theta$ in \eqref{eq:PropRepIntCompl}, we can see that the integral over $B_\theta^c$ is negligible as it is of order $o(1)$. Thus, combining \eqref{eq:SplitIntegral}, \eqref{eq:PropRepIntCompl} and \eqref{eq:PropRepIntWithin} finishes the proof of Theorem \ref{thm:AGM-SLDP} i).\\
	\\

	The proof of Theorem \ref{thm:AGM-SLDP} \textup{ii)} is almost completely the same regarding probabilistic representation, local density estimation and integral approximation, as hardly any of the steps therein use the fact that we are working on $D_{\theta, >}$ for $\theta \in (e^{m_p}, 1)$ instead of $D_{\theta, <}$ for $\theta \in (0, e^{m_p})$, but rather consider a neighbourhood of $\partial D_{\theta, >}$ around $\theta^*$, which coincides with that same neighbourhood of $\partial D_{\theta, <}$  around $\theta^*$, and are therefore the same in both settings. The only notable difference is that one shows the fact that $\theta^*$ minimizes $\Lambda_p^*$ not only on  $\partial D_{\theta, <}$, as in \eqref{eq:EindInfLambda*}, but also on $\overline{D}_{\theta, <}$, by using the fact that $\Lambda_p^*$ is strictly decreasing on $(0, e^{m_p})$ instead of it being strictly increasing on $\theta \in (e^{m_p}, 1)$. Beyond this, the proof is the same as for Theorem \ref{thm:AGM-SLDP} \textup{i)} and is hence omitted here. 
\end{proof}

\bibliographystyle{amsplain}

\end{document}